\documentclass[11pt,reqno]{amsart}
\usepackage{amsfonts,amsmath,amssymb}
\newtheorem{thm}{Theorem}[section]

\newtheorem{corollary}[thm]{Corollary}

\usepackage{color}
\usepackage{graphicx}
\usepackage{epstopdf}

\title[SL(2,R)]{Chaos and integrability in $SL(2,\mathbb R)$-geometry}

\author{ A.V. Bolsinov}\address{Department of Mathematical Sciences,
Loughborough University, Loughborough LE11 3TU, UK,  Moscow State University,  Russia and Moscow Center for Fundamental and Applied Mathematics, Russia.}
\email{A.Bolsinov@lboro.ac.uk}

\author{A.P. Veselov}
\address{Department of Mathematical Sciences,
Loughborough University, Loughborough LE11 3TU, UK,  Moscow State University and Steklov Mathematical Institute, Russia}
\email{A.P.Veselov@lboro.ac.uk}

\author{Y. Ye}
\address{Department of Mathematical Sciences,
Loughborough University, Loughborough LE11 3TU, UK}
\email{Y.Ye@lboro.ac.uk}

\begin{document}

\maketitle

\begin{abstract}

The integrability of the geodesic flow on the three-folds $\mathcal M^3$
admitting $SL(2,\mathbb R)$-geometry in Thurston's sense is investigated.
The main examples are the quotients $\mathcal M^3_\Gamma=\Gamma\backslash PSL(2,\mathbb R)$, where $\Gamma \subset PSL(2,\mathbb R)$ is a cofinite Fuchsian group.
We show that the corresponding phase space $T^*M_\Gamma^3$ contains two open regions with integrable and chaotic behaviour with zero and positive topological entropy respectively.

As a concrete example we consider the case of modular 3-fold with the modular group $\Gamma=PSL({2,\mathbb Z})$, when $\mathcal M^3_\Gamma$ is known to be homeomorphic to the complement of a trefoil knot $\mathcal K$ in 3-sphere.   Ghys proved a remarkable fact that the lifts of the periodic geodesics to the modular surface to $\mathcal M^3_\Gamma$ produce the same isotopy class of knots, which appeared in the chaotic version of the celebrated Lorenz system and were extensively studied by Birman and Williams. We show that in the integrable limit of the geodesic system on $\mathcal M^3_\Gamma$ they are replaced by the simple class of cable knots of trefoil. 
\end{abstract}

\tableofcontents

\section{Introduction}

The general geometrisation programme can be vaguely formulated as follows. Given a manifold, what is the ``best metric'' one can introduce on it? 
To make this precise one has to specify what is assumed from a manifold and what do we mean by the best metric.

One of the major achievements of XIX century mathematics (due to Klein and Poincare, but completed by Koebe in 1907) was the celebrated 
Uniformisation Theorem, claiming that {\it every conformal class of the surface metrics admits complete constant curvature representative.} In particular, on a compact surface we can introduce conformally equivalent metric of positive constant curvature if it is a topological sphere, a flat metric if it is a torus and a negative constant curvature metric if it has genus more than 1.

In dimension three the situation is more complicated. According to  the famous Thurston's {\it Geometrisation Conjecture} \cite{Th}
(now proved by Perelman, see \cite{MT}), any compact orientable 3-manifold can be cut in a special way into pieces, admitting one of the 8 special geometric structures, namely the Euclidean $E^3$, spherical $S^3$ and hyperbolic $\mathbb{H}^3$, the product
type $S^2 \times {\mathbb R}$ and $\mathbb{H}^2 \times {\mathbb R}$ and three
geometries related to the 3D Lie groups: $Nil$, $Sol$ and $\widetilde{SL(2,\mathbb R)}$, where the last group is the universal covering of $SL(2,\mathbb R).$
Corresponding metrics are locally homogeneous, which means that any two points $x$ and $y$ have isometric neighbourhoods $U$ and $V$, see Scott's review \cite{Scott}.


Let $(\mathcal M^3, g)$ be a compact Riemannian 3-fold admitting one of these geometries, and consider the corresponding geodesic flow.
What can we say about its integrability?

In dimension two the answer is well-known: the geodesic flows on round sphere and flat tori are integrable (in any sense), while on surfaces of genus $g>1$ we have the chaotic behaviour with positive entropy.

In this paper we will study the integrability problem for the geodesic flows on 3-folds with $SL(2,\mathbb{R})$-geometry. Instead of the simply-connected group $\widetilde{SL(2,\mathbb R)}$ (which has no finite-dimensional faithful matrix representation) we will use its quotients: the standard matrix Lie group $SL(2,\mathbb{R})$ and $PSL(2,\mathbb{R})=SL(2,\mathbb{R})/\pm I.$

The situation in the other two group cases $Nil$ and $Sol$ were studied earlier by Butler \cite{Butler} and Bolsinov and Taimanov \cite{BT}.
In particular, in \cite{BT} it was shown that in $Sol$-case, on a certain critical level the geodesic flow can be described by a chaotic hyperbolic map of 2D torus, while outside of it we have usual Liouville integrability. As a corollary, this gave the first class of examples of  Liouville integrable (in smooth category) systems with positive topological entropy.

In this paper we show that in the $SL(2, \mathbb R)$-case the chaos spreads out from the critical level to occupy an open region in the phase space.

More precisely, we consider the class of manifolds $\mathcal M^3_\Gamma=\Gamma\backslash PSL(2,\mathbb R)$, where $\Gamma \subset PSL(2,\mathbb R)$ is a finitely generated Fuchsian group acting on the hyperbolic plane $\mathbb{H}^2,$ and $PSL(2,\mathbb R)$ is supplied with a certain class of left-invariant metrics. We will assume that the quotient $\mathcal M^2_\Gamma=\Gamma\backslash \mathbb{H}^2$ is either compact, or at least has a finite area. A particular example is the modular group $\Gamma=PSL(2,\mathbb Z)$, which we discuss in more detail.

Topologically $\mathcal M^3_\Gamma=S\mathcal M^2_\Gamma$ is the unit tangent bundle of the surface $\mathcal M^2_\Gamma$ and carries out a class of natural metrics, coming from the left-invariant metrics on $SL(2,\mathbb{R})$, which are also right $SO(2)$-invariant. They are particular case of the two-parameter family of naturally reductive metrics \cite{HI} on $SL(n,\mathbb R),$ which are left $SL(n,\mathbb R)$-invariant and right $SO(n)$-invariant and determined by the following inner product on the Lie algebra $\mathfrak{g}= sl(n,\mathbb{R}):$
\begin{equation}
\label{metric}
   \langle X,Y \rangle=\alpha(\operatorname{sym} X,\operatorname{sym} Y)+\beta(\operatorname{skew} X,\operatorname{skew} Y), \,\,  \alpha>0>\beta.
\end{equation}
Here $(X,Y):=\operatorname{Tr} XY$ is the standard invariant form on $sl(n,\mathbb{R})$, and $X=\operatorname{skew} X+ \operatorname{sym} X$ is the Cartan decomposition of
$X \in sl(n,\mathbb{R}):$ $$\operatorname{skew} X :=(X-X^{\top})/2 \in so(n), \,\, \operatorname{sym} X:=(X+X^{\top})/2.$$  

They are also known as the {\it generalised Sasaki metrics} \cite{Nagy}, Kaluza-Klein metrics \cite{Mont1995} and appeared in the theory of elastoplasticity \cite{Mielke} (see more details in the next section). In particular, we will see that the Sasaki metric \cite{Sasaki} corresponds to the most convenient case $\alpha=-\beta=2.$

To write down the equations of geodesic flow,  introduce the angular velocity $\Omega:=g^{-1}\dot g \in \mathfrak g$ and the momentum $M=A(\Omega)\in \mathfrak g^*$ determined by the relation $$(\Omega,A(\Omega))=\langle\Omega,\Omega\rangle.$$ Identifying $ \mathfrak g^*$ with $\mathfrak g$ using the Cartan-Killing form, we can write
$$M=\frac{1}{2}(\alpha+\beta)\Omega+\frac{1}{2}(\alpha-\beta)\Omega^{\top}$$ and the corresponding Euler equations \cite{Arnold} as
$$
\dot M=[M,\Omega]=\frac{\beta-\alpha}{2\alpha\beta}[M,M^{\top}].
$$
These equations can be easily integrated and the corresponding geodesics on the group can be found explicitly (see e.g. \cite{HI, Mielke} and Section 3 below).

When $n=2$ introduce the notations by writing the momentum as 
\begin{equation}
\label{abc}
M =\alpha\left(\begin{array}{cc}
a & b \\
c & -a \\
\end{array}
\right)\in \mathfrak g^*\approx \mathfrak g.
\end{equation}
 We have two obvious integrals of the Euler equations: the Hamiltonian  
$$
H=\frac{1}{2}(\Omega,M)=\frac{\alpha}{4\beta}(\beta[4a^2+(b+c)^2]-\alpha (b-c)^2)
$$
and the Casimir function $\Delta=-\frac{1}{\alpha^2}\det M=a^2+bc.$ Their natural extension to the cotangent bundle $T^*G, \, G=PSL(2,\mathbb R)$ by the left shifts (which we will denote by the same letters) give two Poisson commuting integrals of the geodesic flow. As the third, required for the Liouville integrability, integral we can take any non-constant right-invariant function $F$ on $T^*G.$

The situation is different for the quotients $\mathcal M^3_\Gamma=\Gamma\backslash PSL(2,\mathbb R)$, since in general $F$ is not invariant under $\Gamma$ acting from the left.
It turns out that the third Poisson commuting integral, required for  the Liouville integrability, exists only in an open half of the phase space.

More precisely, we prove the following

 \begin{thm} Geodesic flow on $T^*\mathcal M_\Gamma^3$ is integrable in analytic Liouville's sense in the open region of the phase space with $\Delta<0.$ 

 In the region with $\Delta>0$ there are no smooth integrals independent from $H$ and $\Delta$ and the system has positive topological entropy. \end{thm} 
 
This chaos-integrability split has a natural geometric explanation related to the action of the Fuchsian group $\Gamma$ on the co-adjoint orbits $\{\Delta=\delta\}$ of $G$ (see Fig. 2 below). Namely, when $\delta<0$ we have two-sheeted hyperboloid, giving a model of the hyperbolic plane, where the action of $\Gamma$ is discrete, while when $\delta>0$ we have one-sheeted hyperboloid with dense orbits of $\Gamma$ (see more detail in Section 5).
 
 
This kind of split is well known in the classical theory of magnetic geodesic flow on $\mathcal M^2_\Gamma$, see Hedlund \cite{Hedlund-1}, Arnold \cite{Arnold1961}, Paternain et al \cite{BP,Pat}, Taimanov \cite{T2004}. We show that this is not accidental since the magnetic geodesic flow on $\mathcal M^2_\Gamma$ can be considered as a projection of the geodesic flow on $\mathcal M^3_\Gamma.$

More precisely, we show that the natural projection of the geodesics on $SL(2,\mathbb R)$ for our class of metrics to the quotient $SL(2,\mathbb R)/SO(2)$ with constant negative Gaussian curvature $K=-2\alpha^{-1}$ are the curves with constant geodesic curvature $\kappa$ such that
\begin{equation}
\label{curvature}
\mathcal C=\frac{\kappa^2}{K^2}=\frac{(b-c)^2}{4a^2+(b+c)^2}=\frac{\beta H-\alpha\beta \Delta}{\beta H-\alpha^2 \Delta},
\end{equation}
which is an integral of the system. 


It is well-known that in the upper half-plane model these curves are circles if $\mathcal C>1$, and arcs of circles if $\mathcal C<1$ (see Fig. 1). These curves can also be interpreted as magnetic geodesics on hyperbolic plane (see e.g. \cite{Arnold,Hedlund-1}). Note that the condition $\mathcal C>1$ is equivalent to $\Delta<0.$

As a concrete example we consider the case of the modular 3-fold $\mathcal M^3_\Gamma$ with $\Gamma=PSL(2,\mathbb Z)$. It was probably Quillen, who was the first to make an important observation that the quotient
\begin{equation}
\label{quillen}
PSL(2,\mathbb Z)\backslash PSL(2,\mathbb R)=S^3\setminus \mathcal K
\end{equation}
is topologically equivalent to the complement of the trefoil knot $\mathcal K$ in 3-sphere (see Milnor \cite{Milnor}). 

Ghys used this and Birman-Williams results \cite{BW} to establish a remarkable relation between periodic geodesics on $\mathcal M^2_\Gamma$ and a special class of knots (called modular) realised by periodic orbits in the Lorenz system \cite{Ghys} (see more details in Section 6).

We will see that these geodesics are special case of the periodic geodesics on $\mathcal M^3_\Gamma$ at the ``most chaotic'' level of the integral $\mathcal C=0$ 
and extend this link to the integrable region of the geodesic flow on $\mathcal M^3_\Gamma.$ 

 \begin{thm} The periodic geodesics on modular 3-fold $\mathcal M_\Gamma^3$ with sufficiently large values of $\mathcal C$  represent trefoil cable knots in $S^3\setminus \mathcal K.$ 
 Any trefoil cable knot of trefoil can be described in such a way. \end{thm} 
 
 Recall that the {\it cable knots} are special satellite knots, which can be described in the following way (see e.g. \cite{Adams}). Take a solid torus with torus knot $K_{p,q}$ on the boundary and tie it up as a non-trivial knot $K.$ When $K$ is the trefoil knot we get the class of knots called the {\it trefoil cables}. 
 
When $\mathcal C$ is large the projection of the corresponding two-dimensional Liouville torus in $T^*\mathcal M^3_\Gamma$ with frequencies $\omega_1=\frac{\beta-\alpha}{2\beta}|b-c|,\,\, \omega_2=\sqrt{-\Delta}$ on $\mathcal M^3_\Gamma$ gives an embedding onto a torus close to the trefoil fibre of the projection $\mathcal M^3_\Gamma \rightarrow \mathcal M^2_\Gamma$. When 
 $$
 \frac{\omega_1}{\omega_2}=\frac{\beta-\alpha}{2\beta}\frac{|b-c|}{\sqrt{-\Delta}}=\frac{p}{q}
 $$
is rational we have a periodic torus filled by $(p,q)$ cable knots of trefoil.

The appearance of the satellite knots in the integrable region might be expected in view of Thurston's classification of all knots as hyperbolic, torus or satellite (see e.g. \cite{Adams}), but in our case these are specifically cable knots of trefoil.  

Note that in the case of principal congruence subgroup $\Gamma_2$, briefly discussed at the end of Section 6, we have simply torus knots, which is probably the most natural class of ``integrable'' knots.
 

%
We finish the paper with the discussion of topological aspects of the problem in relation with known results about topological obstruction to integrability.
In particular, in the case of compact $\mathcal M^2_\Gamma$ of genus $g\geq 2$ we have $\dim H_1(\mathcal M^3_\Gamma,\mathbb R)=2g>3=\dim \mathcal M^3_\Gamma,$ so by the general Taimanov's result \cite{T} there are no analytically integrable geodesic flows on $T^*\mathcal M_\Gamma^3$ with any metric. Our results show that this does not exclude the analytic integrability in a suitable open region of $T^*\mathcal M^3_\Gamma$.

\section{Geometry and topology of $SL(2,\mathbb R)$}

The group $SL(2,\mathbb{R})$ consisting of $2\times 2$ real matrices with determinant 1 has two relatives: its quotient by the centre $PSL(2,\mathbb{R})=SL(2,\mathbb{R})/\{\pm I\}$ and the universal covering $\widetilde{SL(2,\mathbb{R})}.$ 

It acts  by the isometries $z \rightarrow \frac{az+b}{cz+d}$ of the hyperbolic plane $\mathbb{H}^2$ realised as the upper half-plane $\{ z=x+iy, y>0\} $ with negative constant curvature metric $ds^2=\frac{dx^2+dy^2}{y^2}.$ Since $\pm I$ act trivially, the actual hyperbolic isometry group is $PSL(2,\mathbb{R}).$ The stabiliser of a point $z=i$ is the subgroup $PSO(2)\subset PSL(2,\mathbb{R})$, so $PSL(2,\mathbb{R})$ can be naturally identified with the unit tangent bundle of hyperbolic plane $\mathbb{H}^2=PSL(2,\mathbb{R})/PSO(2).$

Explicitly we have the identification
\begin{equation}
\label{UH}
g=\pm\left(\begin{array}{cc}
	a & b \\
	c & d \\
	\end{array}\right) \in PSL(2, \mathbb R) \longrightarrow \left(z=\frac{ai+b}{ci+d}, \, \xi=\frac{i}{(ci+d)^2}\right) \in S\mathbb H^2,
\end{equation}
where $\mathbb H^2$ is realised as the upper half-plane $\{z=x+iy, \, y>0\} $ with the hyperbolic metric
$$
ds^2=\frac{dzd\bar z}{y^2}.
$$
Indeed, it is easy to check that $\xi \in T_z\mathbb H^2$ has the unit norm in this metric.
Denoting the argument of $\xi$ as $\varphi$ we can introduce convenient coordinates $x,y,\varphi$ on $PSL(2, \mathbb R)$ with $(x,y)\in \mathbb H^2, \varphi \in S^1=\mathbb R/2\pi\mathbb Z.$

Thus both $SL(2,\mathbb{R})$ and $PSL(2,\mathbb{R})$ topologically are open solid torus $\mathbb H^2 \times S^1$ with the fundamental group $\mathbb Z.$ Their universal cover $\widetilde{SL(2,\mathbb{R})}$ is a simply-connected 3-dimensional Lie group, which is known to have no faithful matrix representations.

%

In Thurston's approach the corresponding model geometry is considered on the universal cover $\widetilde{SL(2,\mathbb{R})}$ of $SL(2,\mathbb{R})$, but for our purposes it is enough to consider $SL(2,\mathbb{R})$ itself and its quotients by discrete subgroups. 

As for the metrics we can choose any left-invariant metric on the group, but we will choose the special class, which also right-invariant under the subgroup $SO(2).$ We cannot choose bi-invariant metric since it is known to be not positive definite.

There is a two parameter family of metrics (\ref{metric}), which up to a multiple are determined by the quadratic form 
\begin{equation}
\label{metric2}
|\Omega|^2=4(u^2+vw)+k(v-w)^2, \quad k=1-\frac{\beta}{\alpha} >1
\end{equation}
on the Lie algebra
$$\Omega=\left(\begin{array}{cc}
u & v \\
w & -u \\
\end{array}
\right)\in sl(2,\mathbb R).$$
We have chosen the normalisation $\alpha=2$ to make the Gaussian curvature $K=-2/\alpha$ of the quotient  $SL(2,\mathbb{R})/SO(2)$  to be precisely $-1$ (see below). 

The coordinates $u,v,w$ are related to $a,b,c$ of the corresponding momentum (\ref{abc})  by
\begin{equation}
\label{rela}
u=a, \,\, v=\frac{(2-k)b-kc}{2(1-k)}, \,\, w=\frac{(2-k)c-kb}{2(1-k)},
\end{equation}
\begin{equation}
\label{rela2}
a=u, \,\, b=\frac{(2-k)v+kw}{2}, \,\, c=\frac{(2-k)w+kv}{2}.
\end{equation}
In particular, when $k=2$ we have particularly simple relation
$$
\Omega=\frac{1}{2}M^{\top}=\left(\begin{array}{cc}
a & c \\
b & -a \\
\end{array}
\right).
$$

Direct calculation shows that in the coordinates $x,y,\varphi$ this metric has the form
\begin{equation}
\label{metric3}
ds^2=\frac{dx^2+dy^2}{y^2}+(k-1)\left(d\varphi+\frac{dx}{y}\right)^2.
\end{equation}
This means that the projection $PSL(2,\mathbb{R})\rightarrow PSL(2,\mathbb{R})/PSO(2)$ is a Riemannian submersion of $PSL(2,\mathbb{R})$ with metric (\ref{metric2}) onto the upper half-plane with the standard hyperbolic metric $ds^2=\frac{dx^2+dy^2}{y^2}$ of Gaussian curvature $K=-1.$

One can check that when $k=2$ the metric (\ref{metric2}) is precisely the {\it Sasaki metric} \cite{Sasaki} on $S\mathbb H^2,$ so our class of metrics coincides with the class of the generalised Sasaki metrics considered by Nagy \cite{Nagy}.

As it was shown by Nagy, the projection of the corresponding geodesics to $\mathbb H^2$ are the curves of constant geodesic curvature, which are either circles,  or arcs of them lying in the upper half-plane. We are going to make this more explicit in the next two sections using the Euler-Poincare description of the geodesic flow.

%
%
%

\section{Geodesics on $SL(n,\mathbb R)$ with naturally reductive metrics}

Let $G$ be a semi-simple Lie group $G$ with left-invariant metric defined by an inner product $\langle~,~\rangle$ on the Lie algebra $\mathfrak g.$
We identify $\mathfrak g$ with its dual $\mathfrak g^*$ using the Cartan-Killing form $(~,~).$

The Euler-Poincare equations of the corresponding geodesic flow have the following form (see Arnold \cite{Arnold})
\begin{equation}
  \label{euler's eq}
\dot M=[M,\Omega],
\end{equation}
where $\Omega:=g^{-1}\dot g \in \mathfrak g$ and the momentum $M\in \mathfrak g^*=\mathfrak g$ is determined by the relation $(\Omega,M)=\langle\Omega,\Omega\rangle.$

In the case of naturally reductive metrics (\ref{metric}), which are left $G$-invariant and right $K$-invariant on $G=SL(n,\mathbb R)$ with $K=SO(n)$, we have
\begin{equation}
\label{momentum}
2M=\alpha(\Omega+\Omega^{\top})+\beta(\Omega-\Omega^{\top})
           =(\alpha+\beta)\Omega+(\alpha-\beta)\Omega^{\top}.
\end{equation}
We have also that $2M^{\top}=(\alpha+\beta)\Omega^{\top}+(\alpha-\beta)\Omega$ and thus
\begin{equation}
\label{omega}
\Omega=\frac{(\alpha+\beta)M+(\beta-\alpha)M^{\top}}{2\alpha\beta}.
\end{equation}
Substituting this into (\ref{euler's eq}) we have
\begin{equation}
  \label{EEq}
\dot M=\frac{\beta-\alpha}{2\alpha\beta}[M,M^{\top}].
\end{equation}
Note that
$$
\dot M^{\top}=\frac{\beta-\alpha}{2\alpha\beta}[M,M^{\top}]^{\top}
        =\frac{\beta-\alpha}{2\alpha\beta}[M,M^{\top}]
        =\dot M,
$$
which gives the conservation law $\dot M-\dot M^{\top}\equiv 0$, related to $SO(n)$-invariance of the metric.

Thus in terms of $\Omega$ the equations are
\begin{equation}
  \label{EEqomega}
  \dot\Omega=\frac{\alpha-\beta}{2\alpha}[\Omega^{\top}, \Omega]=\frac{k}{2}[\Omega^{\top}, \Omega],
\end{equation}
where as before $k=1-\frac{\beta}{\alpha}.$

These equations can be easily integrated in the following way (see e.g. \cite{HI,Mielke}.

Note first that 
$$
M=\alpha \, \operatorname{sym} \Omega +\beta\,  \operatorname{skew}\Omega=\alpha[\Omega-k \, \operatorname{skew}\Omega], \,\, \operatorname{skew}\Omega=(\Omega-\Omega^{\top})/2
$$
and introduce the matrices
\begin{equation}
\label{XY}
X=\frac{1}{\alpha}M=\Omega-k \, \operatorname{skew}\Omega \in sl(n), \quad Y= k \, \operatorname{skew}\Omega \in so(n)
\end{equation} 
with $X+Y=\Omega.$ In terms of variables (\ref{abc}) we have
\begin{equation}
\label{XYabc}
          X=\left(
            \begin{array}{cc}
              a & b \\
              c & -a \\
            \end{array}
          \right), \,\, 
Y=\frac{\alpha-\beta}{2\beta}\left(
            \begin{array}{cc}
              0 & b-c \\
              c-b & 0 \\
            \end{array}
          \right)=\frac{k}{2(k-1)}\left(
            \begin{array}{cc}
              0 & c-b \\
              b-c & 0 \\
            \end{array}
          \right).
\end{equation} 

\begin{thm} \cite{HI,Mielke}
The geodesic of naturally reductive metric (\ref{metric}) on $SL(n,\mathbb R)$ with $g(0)=g_0, \, \Omega(0)=\Omega_0$ can be given explicitly by
\begin{equation}
\label{geodesics}
 g(t)=g(0)e^{tX_0}e^{tY_0}
\end{equation}
with $X_0,Y_0$ computed from $\Omega_0$ using formula (\ref{XY}).
\end{thm}

\begin{proof}
We can check by direct calculation that the Euler-Poincare equations (\ref{EEqomega}) are satisfied. We have
$$  \dot g=g_0e^{tX_0}X_0e^{tY_0}+g_0e^{tX_0}Y_0e^{tY_0}=g_0e^{tX_0}\Omega_0e^{tY_0},$$
so $\Omega=g^{-1}\dot g=e^{-tY_0}\Omega_0e^{tY_0}.$ Now
$$ \dot \Omega=e^{-tY_0}\Omega_0Y_0e^{tY_0}-e^{-tY_0}Y_0\Omega_0e^{tY_0}=e^{-tY_0}[\Omega_0,Y_0]e^{tY_0}=[e^{-tY_0}\Omega_0e^{tY_0}, e^{-tY_0}Y_0e^{tY_0}].
$$
Since $e^{tY_0} \in SO(n)$ we have $e^{-tY_0}Y_0e^{tY_0}=\frac{k}{2}e^{-tY_0}(\Omega_0-\Omega_0^{\top})e^{tY_0}=\frac{k}{2}(\Omega-\Omega^{\top}).$
Thus we have 
$$\dot \Omega=[\Omega, \frac{k}{2}(\Omega-\Omega^{\top})]=\frac{k}{2}[\Omega^{\top},\Omega],$$
which coincides with (\ref{EEqomega}).
\end{proof}

Consider the symmetric space $X_n=G/K=SL(n,\mathbb R)/SO(n)$ (of type $AI$ in Cartan's classification \cite{Helgason}) and the natural projection  $\pi: G \rightarrow X_n$.  

\begin{corollary}
The projection of the geodesics on $SL(n,\mathbb R)$ with the naturally reductive metrics to the symmetric space $X_n$ have constant geodesic curvature. 
\end{corollary}

\begin{proof}
Recall that the geodesic curvature $\kappa$ of a curve $\gamma(s)$  parametrised by the arc length $s$  in a Riemannian manifold $M$ is defined as the norm of the covariant derivative of the velocity vector field in the Levi-Civita connection
$$
\kappa =||\frac{D\dot \gamma}{ds}||
$$
(see e.g. \cite{DC}).
From the formula (\ref{geodesics}) we see that, since $e^{tY_0} \in SO(n)$ acts trivially on $X_n$, the projections of the geodesics are orbits of the one-parametric groups $e^{tX_0} \in SL(n,\mathbb R).$
Since $SL(n,\mathbb R)$ acts on $X_n$ by isometries, the geodesic curvature is constant along the orbits. 
\end{proof}

\section{$SL(2,\mathbb R)$ case and hyperbolic magnetic geodesics}

For $n=2$, the space $SL(2,\mathbb R)/SO(2)$  is the hyperolic plane $\mathbb H^2$, so the projection of the geodesics are constant geodesic curvature curves in $\mathbb H^2$.
In the Poincare model  on the upper half-plane these curves are known to be either usual circles, or their arcs lying in the upper half-plane, depending on whether the geodesic curvature $\kappa$ is larger or smaller than the Gaussian curvature $K$ (see e.g. Hedlund \cite{Hedlund-1} and Arnold \cite{Arnold}).

To make all this more explicit, let us use the formulae (\ref{geodesics}) for the geodesics on $SL(2,\mathbb R)$ with metric (\ref{metric2}), assuming for convenience that $k=2$ (which corresponds to the Sasaki metric). 
In that case $$X=\Omega-2 \, \operatorname{skew}\Omega=\Omega^{\top}=\frac{1}{2}M, \quad Y=\Omega-\Omega^{\top},$$ 
or, explicitly, in terms of variables (\ref{abc})
\begin{equation}
\label{omega}
\Omega=\left(
            \begin{array}{cc}
              a & c \\
              b & -a \\
            \end{array}
          \right), \,\, 
          X=\left(
            \begin{array}{cc}
              a & b \\
              c & -a \\
            \end{array}
          \right), \,\, 
Y=\left(
            \begin{array}{cc}
              0 & c-b \\
              b-c & 0 \\
            \end{array}
          \right).
          \end{equation}
          
Note that $X^2=\Delta I, \,\, Y^2=-(b-c)^2I$, where $I$ is the identity matrix and $\Delta=a^2+bc,$  so when $\Delta\neq 0$ we have 
$$
e^{Xt}=\cosh \sqrt{\Delta}t\,I+ \frac{\sinh\sqrt{\Delta}t}{\sqrt{\Delta}}X, \, e^{Yt}=\left(
            \begin{array}{cc}
              \cos(b-c)t & -\sin(b-c)t \\
              \sin(b-c)t &  \cos(b-c)t \\
            \end{array}\right).
            $$
When $\Delta=0,$ then $X^2=0$ and $e^{Xt}=I+Xt.$

Substituting this into the equation (\ref{geodesics}), we have the explicit formula for the geodesic on the group $G=SL(2,\mathbb R)$ passing through $g(0)=g_0$ in direction $\dot g(0)=g_0\Omega \in T_{g_0}G$:
\begin{equation}
\label{geod2}
g(t)=g_0(\cosh \sqrt{\Delta}t \,I+ \frac{\sinh\sqrt{\Delta}t}{\sqrt{\Delta}}X)\left(
            \begin{array}{cc}
              \cos(b-c)t & -\sin(b-c)t \\
              \sin(b-c)t &  \cos(b-c)t \\
            \end{array} \right).
\end{equation}

Assume for simplicity that $g_0=I,$ then acting by the right-hand side on $i \in \mathbb H^2$ gives the projection of the geodesic  to $\mathbb H^2$ explicitly as
\begin{equation}
\label{zt}
 z(t)={\frac {i \left( {{\rm e}^{2\,\sqrt {\Delta}t}\sqrt {\Delta}}-ib{{\rm e}^{2\,
\sqrt {\Delta}t}}+{
{\rm e}^{2\,\sqrt {\Delta}t}}a +\sqrt {\Delta}-a+ib \right) }{{{\rm e}^{2\,\sqrt
{\Delta}t}\sqrt {\Delta}}+ic{{\rm e}^{2\,\sqrt {\Delta}t}}-{{\rm e}^{2\,\sqrt {
\Delta}t}}a+\sqrt {\Delta}+a-ci}}.
\end{equation}
Assuming that $c\neq 0$ one can check that
$$
z(t)-\left(\frac{a}{c}+\frac{c-b}{2c}i\right)=-\left(\frac{a}{c}-\frac{b+c}{2c}i \right)\frac {Z(t)}{\overline{Z(t)}} 
 $$
where
$$
Z(t)={{\rm e}^{2\,\sqrt {\Delta}t}\sqrt {\Delta}}-
{{\rm e}^{2\,\sqrt {\Delta}t}}a-ic{{\rm e}^{2\,
\sqrt {\Delta}t}}+\sqrt {\Delta}+a+ic.
$$
Thus 
$$
\left |z(t)-\frac{2a+(c-b)i}{2c}\right |^2=\left |\frac{2a-(b+c)i}{2c}\right |^2=\frac{4a^2+(b+c)^2}{4c^2},
$$
which gives a Euclidean circle (or the arc of it belonging to the upper half-plane). 
In the case of $c=0$, we have the (part of the) straight line 
$$
z(t)=\frac{b(e^{2\sqrt{\Delta}t}-1)+2aie^{2\sqrt{\Delta}t}}{2a}.
$$

Summarising we have the following result.

\begin{thm}
The projection to $\mathbb{H}^2$  of the geodesic on $SL(2,\mathbb R)$ passing through the identity in the direction $\Omega$ given by (\ref{rela}) with $c\neq 0$ and $\Delta=a^2+bc<0$  is the Euclidean circle centred at $$z_0=(2a+(c-b)i)/2c$$ with radius $$R=\sqrt{4a^2+(b+c)^2}/2c.$$
If $\Delta>0$ the projections are the arcs of these circles lying in the upper half-plane.
When $c=0$ we have the part of the straight line $y=\frac{2a}{b}x+1$ belonging to the upper half-plane.
\end{thm}





The three different cases depending on the sign of $\Delta$ are shown on Fig.1.

\begin{figure}[h]
  \includegraphics[height=30mm,width=48mm]{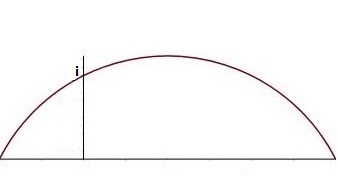}   \includegraphics[width=32mm]{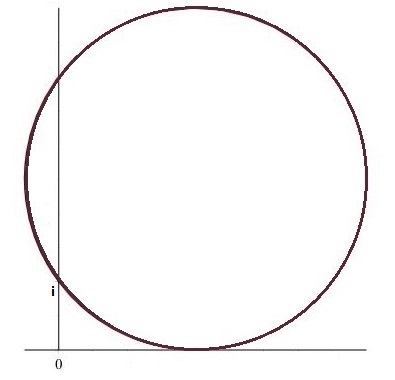}    \includegraphics[width=35mm]{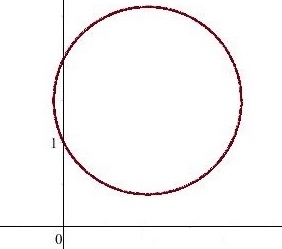}
  \caption{Projection of geodesics  with $\Delta>0$ (left), $\Delta=0$ (middle), $\Delta<0$ (right) respectively.}
  \end{figure}
  
 Since the Euclidian circles are known to be the curves of constant geodesic curvature, it is natural to consider this relation in more detail.
 
Computation of the Christoffel symbols \cite{DFN}
	$$\Gamma^k_{\,\,ij}=\frac12 g^{kl}(g_{li,j}+g_{lj,i}-g_{ij,l})$$
	for the hyperbolic metric $ds^2=\frac{1}{y^2}(dx^2+dy^2)$ gives (with $x^1=x, x^2=y$)
	$$\Gamma^1_{\,\,12}=\Gamma^1_{\,\,21}=\Gamma^2_{\,\,22}=-\frac{1}{y^2}, \, \Gamma^2_{\,\,11}=\frac{1}{y^2},$$
with all other symbols to be zero. The equation for geodesics 
$$
\frac{D\dot x}{ds}=\ddot{x}^k+\Gamma^k_{\,\,ij}\dot{x^i}\dot{x^j}=0
$$
in our case with $z=x+iy$ becomes 
$$
\frac{D\dot z}{ds}=\ddot{z}+\frac{i\dot{z}^2}{y}=0.
$$ 
The equation for curves $z(s)$ of constant geodesic curvature $\kappa$  in the arc length parameter $s$  is
 \begin{equation}
\label{geodcc} 
\frac{d^2 z}{ds^2}+\frac{i}{y}\left(\frac{dz}{ds}\right)^2=i\kappa \frac{dz}{ds},
\end{equation}
where by definition the velocity $||\frac{dz}{ds}||$ (computed in the hyperbolic metric) is equal to 1.

In the parameter $t$ when this velocity $V=||\frac{dz}{dt}||$ is constant (so $s=Vt$), we have the equation
 \begin{equation}
\label{geodV} 
\frac{d^2 z}{dt^2}+\frac{i}{y}\left(\frac{dz}{dt}\right)^2=i\kappa V\frac{dz}{dt},
\end{equation}

Note that the right-hand side can also be interpreted as the Lorenz force in the magnetic field defined by the 2-form 
$Bd\sigma=Bdx\wedge dy/y^2$ with the density $B=\kappa V$, so the same equations describe magnetic geodesics on the hyperbolic plane (see e.g. \cite{Arnold}, \cite{Miranda}).

%

Using the explicit formula (\ref{zt}) one can check that the projection of the corresponding geodesic  (\ref {geod2}) has constant velocity 
$
V=\sqrt{4a^2+(b+c)^2}
$
and satisfies the equation (\ref{geodV}) with $$\kappa=\frac{b-c}{\sqrt{4a^2+(b+c)^2}},$$
where we assigned the sign to the geodesic curvature using equation (\ref{geodV}).

\begin{thm}
The projection of the geodesics on $SL(2,\mathbb R)$ with metric (\ref{metric2}) and momentum $M$ given by (\ref{abc}) to $\mathbb H^2$ is a curve of constant geodesic curvature
 \begin{equation}
\label{kappa} 
\kappa=\frac{b-c}{\sqrt{4a^2+(b+c)^2}}
\end{equation}
parametrised by $t=s/\sqrt{4a^2+(b+c)^2}$, where $s$ is the arc length.

Equivalently, it can be described as a magnetic geodesic in the constant magnetic field with density $B=b-c.$
\end{thm}

For the non-normalised metrics (\ref{metric}) the quotient $SL(2,\mathbb R)/SO(2)$ has Gaussian curvature $K=-2\alpha^{-1}$ and $$\kappa=K\frac{c-b}{\sqrt{4a^2+(b+c)^2}},$$
so the ratio
$$\mathcal C:=\frac{\kappa^2}{K^2}=\frac{(b-c)^2}{4a^2+(b+c)^2}$$
is independent on the choice of the metric.

Depending on the values of  $\mathcal C$  the corresponding curves on the hyperbolic plane are called {\it hypercycles} when $\mathcal C<1$,
{\it horocycles} when $\mathcal C=1$ and {\it hyperbolic circles} when $\mathcal C>1$ (see Caratheodory \cite{Caratheodory}).
The special case of hypercycles with $\mathcal C=0$ are the usual geodesics with zero geodesic curvature. 

Hedlund \cite{Hedlund-1} was probably the first to study the properties of these curves on the quotient $\mathcal M^2_\Gamma=\Gamma\backslash \mathbb{H}^2$ of hyperbolic plane by the Fuchsian groups $\Gamma \subset PSL(2,\mathbb R)$. They were studied later by Arnold \cite{Arnold}, see also more recent papers by Paternain et al \cite{BP,Pat}, Taimanov \cite{T2004} and Miranda \cite{Miranda} and references therein.

Recall that the geodesic flow on a surface $\mathcal M^2$ is called {\it transitive} if there exists a geodesic which is everywhere dense on the unit tangent bundle of $\mathcal M^2.$

Let $\Gamma \subset PSL(2,\mathbb R)$ be a {\it cocompact Fuchsian group}, so that the quotient $\mathcal M^2_\Gamma=\Gamma\backslash \mathbb{H}^2$ is compact.
 
 We believe that the following result is true for all {\it cofinite Fuchsian groups} $\Gamma$ with the quotients $\mathcal M^2_\Gamma$ of finite area, but we cannot find a proper reference.


\begin{thm} (Hedlund \cite{Hedlund-1}, Arnold \cite{Arnold})
The magnetic geodesic flow (\ref{geodcc}) on $\mathcal M^2_\Gamma$ is transitive if and only if $\kappa\leq 1.$
The metric entropy $h(\kappa)$ of such flow is $$h(\kappa)=\sqrt{1-\kappa^2}.$$
\end{thm}

\begin{corollary}
The magnetic geodesic flow with $\kappa <1$ has the positive Lyapunov exponent 
 \begin{equation}
\label{lambda} 
\lambda=\sqrt{1-\kappa^2}.
\end{equation}
\end{corollary}

Recall that the {\it Lyapunov exponent} of the flow $f_t, \, t \in \mathbb R$ on a Riemannian manifold $M$ can be defined (see e.g. \cite{Pesin}) as 
$$
\lambda(f)=\limsup_{t\to+\infty}\frac{1}{t}\ln||df_t(x)||,
$$
which in ergodic case is a constant independent on $x \in M$, describing the rate of separation of infinitesimally close orbits. In our case we can apply Pesin's formula \cite{Pesin}, which says that $$\lambda(f)=h(f)=\sqrt{1-\kappa^2}>0.$$ 

The general hypercyclic case $\kappa<1$ can be reduced to the geodesic flow case with $\kappa=0$ since the hypercycles are known to be the curves equidistant from the geodesics (see \cite{Arnold, Hedlund-1}). The most subtle horocyclic case $\kappa=1$ was sorted by Hedlund in \cite{Hedlund-1}. In the case $\kappa>1$ all the orbits are periodic.

Taimanov \cite{T2004} interpreted these results in terms of the integrability of the corresponding magnetic geodesic flow on $\mathcal M^2_\Gamma$. Indeed, since the geodesic curvature of the magnetic geodesics is $\kappa=B/V$, where $B$ is the density of the magnetic field and $V$ is the velocity, the integrability condition $\kappa>1$ is equivalent to the condition $V<B.$

We will use these results now to study the geodesic flow on three-folds $\mathcal M^3_\Gamma=\Gamma\backslash PSL(2,\mathbb R).$

\section{Geodesic flow on Fuchsian quotients of $PSL(2,\mathbb R)$}

Let $\Gamma\subset G = PSL(2,\mathbb R)$ be a finitely generated Fuchsian group with the compact quotient $\Gamma\backslash \mathbb H^2=\mathcal M_\Gamma^2$ and consider $SL(2,\mathbb R)$-manifold 
    $$\mathcal M_\Gamma^3=\Gamma\backslash G.$$
It is known (see e.g. Thurston \cite{Th}) that in the compact case 
    $$\mathcal M_\Gamma^3=S\mathcal M_\Gamma^2 \subset T\mathcal M_\Gamma^2$$
    is the unit tangent bundle of the corresponding surface $\mathcal M_\Gamma^2$ of genus $g\geq 2.$
    In the non-compact case (for example, for $\Gamma=PSL(2,\mathbb Z)$) the quotient could be an orbifold, so one should be a bit more careful here.


Let $\Omega=g^{-1}\dot g \in \mathfrak g$, $M=A(\Omega) =\alpha\left(
            \begin{array}{cc}
              a & b \\
              c & -a \\
            \end{array}
          \right) \in \mathfrak g^*$ be the left angular velocity and left momentum as before. Note that $a,b,c$ can be considered as  left invariant functions on $T^*SL(2,\mathbb R).$
          
Introduce now the right momentum 
$$
m=gMg^{-1}=\alpha\left(
            \begin{array}{cc}
              u & v \\
              w & -u \\
            \end{array}
          \right) \in \mathfrak g^*.
$$
It is well-known (see Arnold \cite{Arnold}) that $m$ is preserved: $\dot m=0,$ which can be also easily checked directly using Euler-Poincare equations (\ref{euler's eq}).

The matrix elements $u,v,w$ of the right momentum are linear functions on $\mathfrak g^*$ and thus are right-invariant functions on $T^*G.$

In terms of these functions, using relation (\ref{omega}), the Hamiltonian of the geodesic flow on $T^*\mathcal M^3_\Gamma=\Gamma\backslash T^*G$ for the metric (\ref{metric}) can be written as
 \begin{equation}
\label{H}
H=\frac12 (M,\Omega)=\frac{\alpha}{4\beta}(\beta[4a^2+(b+c)^2]-\alpha (b-c)^2),
\end{equation}
which is well defined function on $T^*\mathcal M^3_\Gamma$. 

Another obvious integral is the Casimir function of the corresponding Poisson bracket on $\mathfrak g^*$
$$
\Delta=-\det M=-\det m=a^2+bc=u^2+vw.
$$
For Liouville integrability we need one more analytic integral. On the group $G$ we can take, for example,
$F=v-w$ corresponding to the left $SO(2)$ invariance of the metric, but it is not $\Gamma$-invariant.

To find the invariants we need to study the action of $\Gamma \subset G$ on the space of the conserved momenta $m$, which is $\mathfrak g^*.$ The action of the group $G$
preserves $\Delta$, which defines the pseudo-Euclidean structure on $\mathfrak g^*$ and thus an isomorphism of $PSL(2,\mathbb R)$ with the identity connected component $SO(2,1)_0$ of the group $SO(2,1).$

The cone $\Delta=u^2+vw=0$ splits $\mathfrak g^*$ into two open regions: $\Delta>0$ and $\Delta<0$, foliated by the one- and two-sheeted hyperboloids $\Delta=\delta$ respectively (see Figure 2).

\begin{figure}[h]
  \includegraphics[width=150mm]{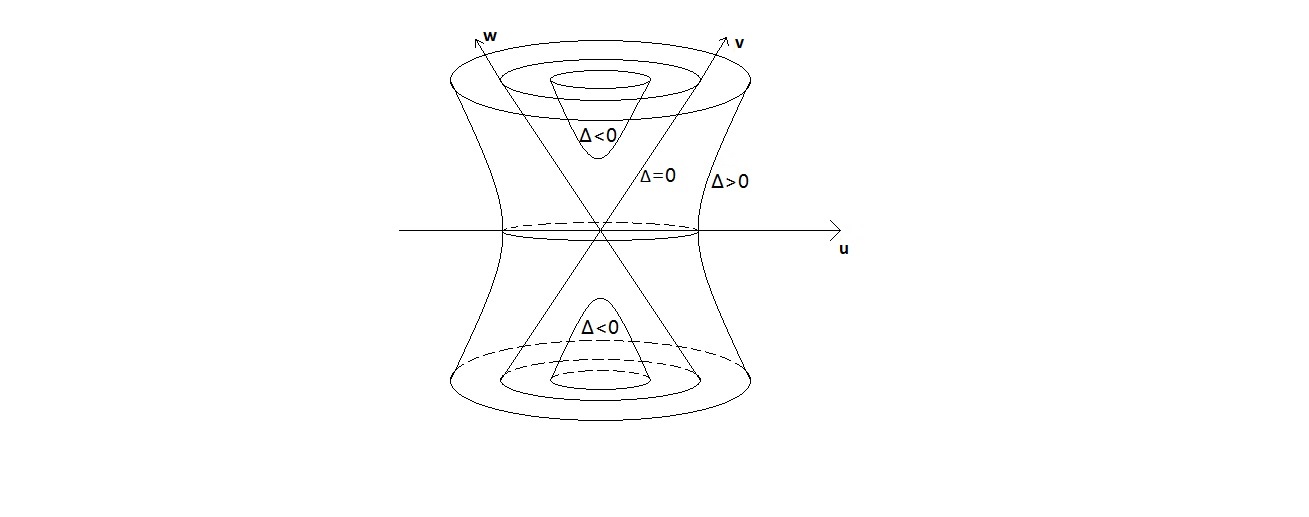}  
  \caption{Symplectic leaves $\Delta=\delta$ of the Poisson bracket on $\mathfrak g^*.$}
\end{figure}
  
 A sheet of two-sheeted hyperboloid $H_\delta$ given by $\Delta=\delta, \delta<0$ represents a well-known model of the hyperbolic plane $\mathbb H^2$, so the action of $\Gamma$ is discrete here.  From the theory of automorphic functions it follows that there exists a non-constant real analytic $\Gamma$-invariant function $F$ on $H_\delta$. We can extend it by homogeneity to a function $F(m)$  defined on the domain $\{\Delta < 0\} \subset \mathfrak g^*$. 
  
In contrast, on one-sheeted hyperboloid $\Delta=\delta, \delta>0$ the generic orbits of $\Gamma$ are known to be dense (see e.g. \cite{Bo}, Section VI, Property 2.12). 
Let us sketch the proof of this, kindly provided to us by John Parker \cite{Parker}.

First, we remind the following correspondence due to Klein \cite{Klein} between the one-sheeted hyperboloid and the double cover of the space of geodesics on the hyperbolic plane.

Recall that Klein's model of the hyperbolic plane is the interior of the conic on the real projective plane. The points outside the conic correspond by polarity to the straight chords inside the conic, which are the lines in Klein's model. Representing this conic by $\Delta=0$ we have the correspondence required.
According to Hedlund \cite{Hedlund-1} almost every hyperbolic geodesic is dense in the quotient, so the same is true for generic orbits of $\Gamma$ on the one-sheeted hyperboloids $H_\delta$ with $\delta>0.$

Summarising we have the following

\begin{thm}
The geodesic flow on $T^*\mathcal M^3_\Gamma$ has no smooth right-invariant integrals $F$ independent from $\Delta$ in the part of the phase space $T^*\mathcal M^3_\Gamma$ with $\Delta\geq 0.$
\end{thm}

We can relax now the assumption that the integral $F$ must be right-invariant proving our main result.

Recall that the topological entropy $h_{top}(\phi)$ of the geodesic flow $\phi_t$ on $M=S\mathcal M$  is defined as
$$
h_{top}(\phi):= \lim_{\epsilon \to 0}\limsup_{T\to+\infty}\frac{\log N(\phi, T,\epsilon)}{T},
$$
where $N(\phi, T, \epsilon)$ be the minimal cardinality of a finite set $X=X(\epsilon)\subset M$ such that for any $v\in M$ there exists $w\in X$ such  that $$\sup\limits_{ 0\leq t\leq T} d(\phi_t v, \phi_{t}w)< \epsilon$$ (see details and history of the notion in Katok's review \cite{Katok}).

The topological entropy satisfies the inequality
$$
h_{top}(\phi)\geq h_{\mu}(\phi)
$$
for metric entropy  $h_mu(\phi)$  defined with respect to any ergodic $\phi$-invariant measure $\mu.$ It has also the following properties (see e.g. \cite{JL}).

Let $\phi_t$ and $\psi_t$ be the flows on compact manifolds $X$ and $Y$ respectively and $f: X \rightarrow Y$ be a map such that $\psi_t\circ f=f \circ \phi_t.$ Then
\begin{itemize}
\item If $f$ is injective then $h_{top}(\phi)\leq h_{top}(\psi),$
\item If $f$ is surjective then $h_{top}(\phi)\geq h_{top}(\psi).$
\end{itemize}

 \begin{thm} Geodesic flow on $T^*\mathcal M_\Gamma^3$ is integrable in analytic Liouville's sense in the open region of the phase space with $\Delta=a^2+bc<0.$ 
 
 In the region with $\Delta>0$ there are no smooth integrals independent from $H$ and $\Delta.$ At the integral level $\mathcal C<1$ the system has positive topological entropy  
 \begin{equation}
\label{ineq}
h_{top}\geq \sqrt{1-\mathcal C}.
 \end{equation}
 \end{thm} 
 
 \begin{proof}
As we have seen, the level  $\Delta=\delta<0$ defines in $\mathfrak g^*$ two-sheeted hyperboloid $H_\delta$, one sheet of which can be identified with the hyperbolic plane.  Let $F$ be as before a real analytic $\Gamma$-invariant function on $H_\delta$ extended to the open region $\{\Delta < 0\} \subset \mathfrak g^*$ by homogeneity, and extend it further to the corresponding domain of $T^*G$ by the right shifts as $F(m).$ Since $F(m)$ is $\Gamma$-invariant this gives us a well-defined function $F$ on the quotient $T^*\mathcal M_\Gamma^3=\Gamma\backslash T^*G.$

It is easy to see that $H$, $\Delta$ and $F$ are 3 independent analytic integrals on $T^*\mathcal M_\Gamma^3$, which are Poisson commuting. 
Thus the geodesic flow is Liouville integrable in this domain, even in the analytic category. Note that we have one more integral by taking an independent from $F$ real analytic $\Gamma$-invariant function (which will not Poisson commute with $F$, of course), so the Liouville tori are two-dimensional, in agreement with the explicit description of geodesics above.

The problem is however that writing down explicitly the automorphic integral $F(m)$ is only possible in special cases (see the next section).  This is related with the non-effectiveness of the solution of the uniformisation problem in the compact hyperbolic case. 
 
When $\Delta>0$ we can prove a stronger result that in fact the geodesic flow has no invariant 3-dimensional Liouville tori with fixed $H$ and $\Delta=\delta>0$. 
Indeed, let us assume that we have such a torus $T^3 \subset T^*\mathcal M_\Gamma^3$ and consider its projection to $\mathcal M_\Gamma^3$.
Because in that case $\mathcal C<1$ the system has positive Lyapunov exponent and the projections diverge exponentially fast, which is impossible since on the Liouville torus they are simply winding lines.

To show that the topological entropy satisfy the inequality (\ref{ineq}) consider the geodesic flows on the unit tangent bundles of $\mathcal M_\Gamma^3$ and $\mathcal M^2_\Gamma$. We have seen that the geodesics on $\mathcal M_\Gamma^3$ at the integral level $\mathcal C=\kappa^2<1$ project into magnetic geodesics on $\mathcal M^2_\Gamma$ with geodesic curvature $\kappa$. Let $X(\kappa) \subset SM_\Gamma^3$ be the union of all these geodesics, and $f$ be its projection onto $Y=S\mathcal M^2_\Gamma.$ Since the corresponding magnetic geodesic flow on $Y$ has metric entropy $h=\sqrt{1-\kappa^2}=\sqrt{1-\mathcal C}$, by the properties of the topological entropy  for the geodesic flow on $X(\kappa)$ we have the inequality (\ref{ineq}).   
 \end{proof}
 
 To derive now the Theorem 1.1 for the cofinite Fuchsian groups $\Gamma$ with the non-compact quotients $M^2_\Gamma$ we can use the results by Gurevich and Katok \cite{GK}, who proved that the (properly understood) topological entropy of the geodesic flow on $SM^2_\Gamma$ in this case also equals 1 (see Theorem 12 in \cite{GK}).

 \section{Geodesics on the modular 3-fold and knots}
 
 Consider now the special modular case of the group $\Gamma=PSL(2, \mathbb Z)$. In that case the quotient $\mathcal M^2=PSL(2, \mathbb Z)\backslash \mathbb H^2$ is an  orbifold with two orbifold points $i$ and $e^{i\pi/3}$ corresponding to the elliptic elements in $PSL(2, \mathbb Z)$ of order 2 and 3 respectively. 

The modular surface is naturally the moduli space of elliptic curves with the classical function (known as {\it Hauptmodul})  $j(\tau): \mathcal M^2 \rightarrow \mathbb C,$ establishing its equivalence to $\mathbb C.$

\begin{figure}[h]
\includegraphics[width=50mm]{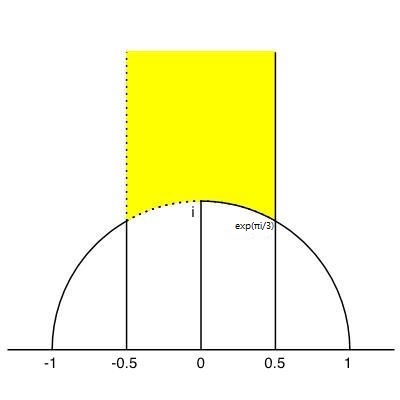}  \quad  \includegraphics[width=50mm]{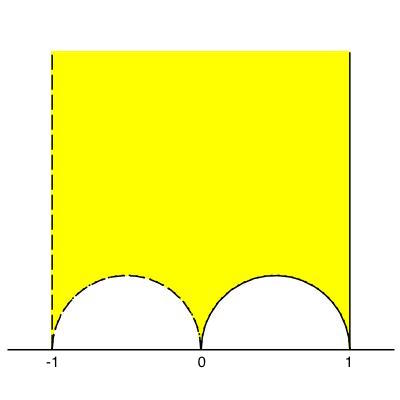}
\caption{The fundamental domains of $\Gamma$ and $\Gamma_2$}
\end{figure}

The geodesics on modular surface $\mathcal M^2$ were studied in detail by Artin in the seminal paper \cite{Artin}, who used the symbolic dynamics and the theory of continued fractions to describe them. 

In particular, Artin showed that the periodic geodesics on the modular surface correspond to the lines in Klein's model determined by the equations with integer coefficients, which are invariant under suitable hyperbolic element $g \in PSL(2,\mathbb Z).$
Any such element preserves an indefinite  binary quadratic form $Q$ with integer coefficients, so we have one-to-one correspondence between periodic geodesics on $\mathcal M^2$ and the set of integer indefinite binary quadratic forms, considered up to proportionality.

Consider now the modular 3-fold $\mathcal M^3=PSL(2, \mathbb Z)\backslash PSL(2, \mathbb R)$. 
In this case we can produce the integrals in the domain of $S\mathcal M^3$ with $\Delta<0$, additional to $H$ and $\Delta,$ explicitly by taking the real and imaginary parts of the Hauptmodul $j$ considered as the function on the one sheet of the hyperboloid $\{\Delta =-1\} \subset \mathfrak g^*$ identified with the upper half-plane $\mathbb H^2.$

The non-integrability of the geodesic flow in the domain with $\Delta>0$ can be shown using the same arguments as before with the positivity of the topological entropy following from the results of Gurevich and Katok \cite{GK}. 

 The important observation, usually attributed to Quillen (see Milnor \cite{Milnor}), is that topologically  
 \begin{equation}
\label{quil}
\mathcal M^3\cong S^3\setminus \mathcal K
\end{equation} 
is equivalent to the complement in $S^3$ to a trefoil knot $\mathcal K$, which is a $(2,3)$-torus knot.\footnote {We are very grateful to Graeme Segal for the discussions of this remarkable fact and the history of its discovery.} Note that the fundamental group of this complement is the braid group $B_3$, so that the modular 3-fold is the quotient
$$
\mathcal M^3=\widetilde{SL(2,\mathbb R)}/B_3.
$$

Recall (see e.g.  \cite{Adams}) that a $(p,q)$-{\it torus knot} $K_{p,q}$ is a special type of knots specified by a pair of coprime integers $p$ and $q$. It can be realised on the surface of the solid torus in $\mathbb R^3$, winding $p$ times around the axis of rotation of the torus and $q$ times around the central circle of the solid torus. When $p=2,q=3$ we have the trefoil knot shown on the left of Fig. 3.

\begin{figure}[h]
\includegraphics[width=50mm]{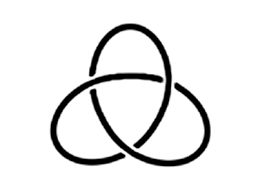} \quad \includegraphics[width=40mm]{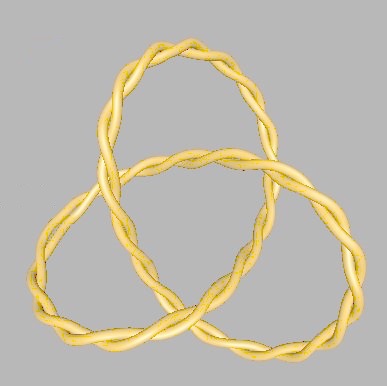}  
\caption{Trefoil knot $\mathcal K$ and its $(2,33)$ cable knot.}
\end{figure}

To prove the homeomorphism (\ref{quil}) note that the quotient $$PSL(2, \mathbb Z)\backslash PSL(2, \mathbb R)=SL(2, \mathbb Z)\backslash SL(2, \mathbb R)$$ can be interpreted as the moduli space of the lattices $\mathcal L$ on the Euclidean plane with fixed area of the quotient, or equivalently the space of elliptic curves $\mathbb C/\mathcal L$ up to real scaling. The corresponding $\wp$-function satisfies the Weierstrass equation 
$$
(\wp')^2=4\wp^3-g_2\wp-g_3
$$
with the discriminant $D=g_2^3-27g_3^2\neq 0.$ The intersection of the unit sphere $S^3 \subset \mathbb C^2(g_2,g_3)$ with the set $D=0$ is a $(2,3)$-torus knot, known as trefoil knot.

Alternatively, one can argue that the natural map $\mathcal M^3 \to \mathcal M^2$ is the Seifert fibration with two singular fibers over orbifold points of order 2 and 3, which implies that the missing fiber over infinity is a $(2,3)$-torus knot. More interesting relations can be found in a nice article \cite{Mostovoy} by Mostovoy.

Note that there is a deep theorem by Gordon and Luecke \cite{GL} saying that any knot in the 3-sphere is determined up to isomorphism by the homeomorphism-type of its complement (it is not true though for links).

The homeomorphism (\ref{quil}) plays a key role in the remarkable work by Ghys \cite{Ghys}, linking closed geodesics on the modular surface $\mathcal M^2$ with periodic orbits in the celebrated Lorenz system \cite{Lorenz}
$$
\begin{cases}
\dot x = \sigma(-x+y)\\
\dot y =  r x -y -xz\\
\dot z = -b z + xy
\end{cases}
$$
depending on positive real parameters $\sigma,r,b,$ where the most essential for us parameter $r$ has the physical meaning of relative Rayleigh number.

Historically this was one of the first computer observation of a chaotic attractor. Lorenz had chosen the values $\sigma=10, b=8/3$ and $r=28$ to observe the peculiar behaviour known now as ``Lorenz butterfly'' shown on the left of Fig. 4.

 \begin{figure}[h]
\includegraphics[width=40mm]{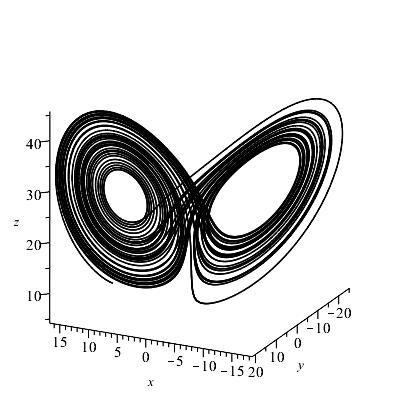}  \, \includegraphics[width=40mm]{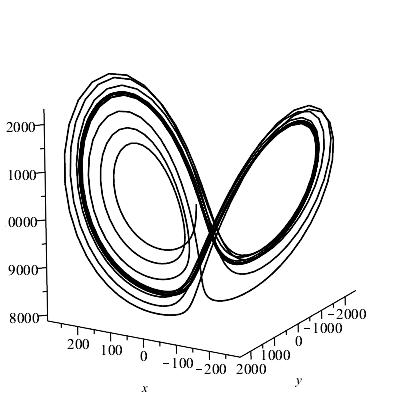} \, \includegraphics[width=40mm]{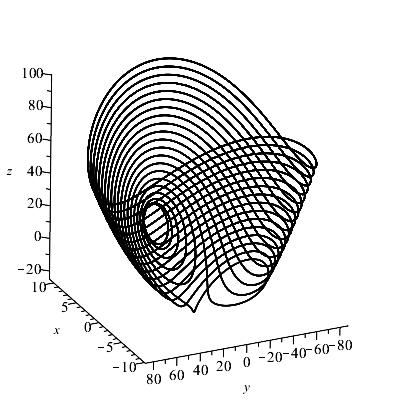} 
\caption{The Lorenz trajectories for $r=28,\,10000$ and $r=\infty.$}
\end{figure}

It is known (see e.g. Strogatz \cite{Strogatz}) that in the formal limit $r\to \infty$ the Lorenz dynamics becomes periodic (see the right of Fig. 4), so the chaos disappears for large $r$  (see Yudovich \cite{Yudovich} {\footnote{We are grateful to Anatoly Neishtadt, who attracted our attention to this not widely known preprint.} and Robbins \cite{Robbins}).  
For large $r$ the Lorenz system can be considered as perturbation of the limiting case $r=\infty$  and has no more than 3 periodic orbits (see \cite{Yudovich, Robbins, MN}).

Birman and Williams \cite{BW} studied periodic trajectories of Lorenz system  in the chaotic region from topological point of view. They observed that these trajectories are knotted in a very special way and studied the corresponding class of knots. In particular, they have shown that all these knots are prime and fibered  and contain as a subclass all the algebraic knots, including all torus knots \cite{BW}. 

Remarkably, as it was proved by Ghys \cite{Ghys}, the same isotopy class of knots appears as canonical lifts to $\mathcal M^3\cong S^3\setminus \mathcal K$ of the periodic geodesics on the modular surface! Some of the lifted geodesics, which can be labelled by the primitive elements of $SL(2, \mathbb Z)$, are shown on Fig. 5, borrowed from \cite{GhL}. \footnote{We are very grateful to Jos Leys for kind permission to use his amazing images.}

 \begin{figure}[h]
\includegraphics[width=55mm]{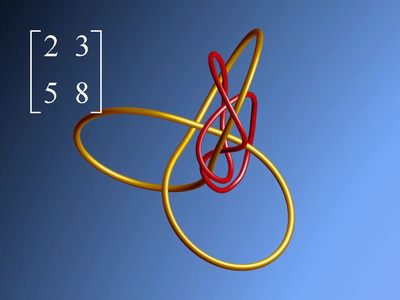} \quad \quad \includegraphics[width=55mm]{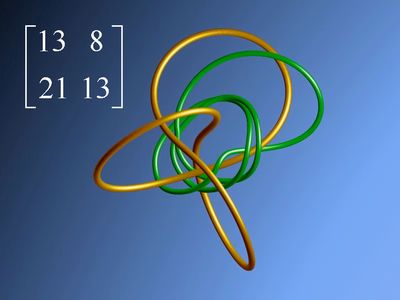} 
\caption{The images of the lifted modular geodesics in the complement of the trefoil knot from \cite{GhL}.}
\end{figure}

Ghys called these lifts {\it modular knots} and proved that their linking number with the trefoil knot $\mathcal K$ is given by the number-theoretic Rademacher function, which allowed Sarnak \cite{Sarnak} to do the refined counting of the modular geodesics by this linking number.

From our perspective all this is about the zero level of the integral $$\mathcal C=\frac{(b-c)^2}{4a^2+(b+c)^2}=0$$ of our geodesic flow on  $S\mathcal M^3$. 
The corresponding geodesics have the form 
$$
g(t) = g_0 e^{tX}, \quad  X = \begin{pmatrix}   a & b \\ b  & -a \end{pmatrix}.
$$
The geodesics considered by Ghys correspond to $a=1$, $b=0$. 
In the general case, the periodic geodesics give equivalent knots, so all periodic geodesics on $\mathcal M^3$ with $\mathcal C = 0$ form modular knots.

Note that $\mathcal C=0$ is the most chaotic level with maximal entropy $h=1.$ Let us consider now the most integrable limit when $\mathcal C\to\infty.$

In that case the projection of the corresponding geodesics on $\mathcal M^3$ are small circles on $\mathcal M^2$, turning to points when $\mathcal C=\infty$, which is the case when
$
a=b+c=0.
$
In that case the geodesics on $\mathcal M^3$ become fibres of the Seifert fibration $\mathcal M^3 \to \mathcal M^2$, which topologically are trefoil knots (as well as  $\mathcal K$, which is the fiber over infinity). In the Lorenz system they correspond to periodic orbits in the limit $r=\infty.$

This means that for large values of $\mathcal C$, the projections of the corresponding Liouville tori are embeddings and knotted as trefoil knots. The frequencies of the motion are
$$
\omega_1=\frac{\beta-\alpha}{2\beta}|b-c|, \quad \omega_2=\sqrt{-\Delta}, \,\, 
$$
where $\omega_1$ is the frequency along the trefoil fibre, while $\omega_2$ is the frequency along the hyperbolic circle. 

In the particular case when the ratio
 \begin{equation}
\label{pq}
\frac{\omega_1}{\omega_2}=\frac{\beta-\alpha}{2\beta}\frac{|b-c|}{\sqrt{-\Delta}}=\frac{p}{q}, \quad p,q \in \mathbb Z,
\end{equation}
is rational, the trajectories are periodic and form a special case of satellite knots of trefoil, called {\it cable knots} \cite{Adams}. They can be realised as torus knots $K_{p,q}$ on the solid torus knotted as a trefoil knot (see Fig. 4). 

Note that by $a,b,c$ as well as parameters of the metric $\alpha, \beta$ we can realise any values of $p,q,$ so any trefoil cable knot can be realised in this way.

 \begin{thm} The periodic geodesics on modular 3-fold $M_\Gamma^3$ with sufficiently large values of $\mathcal C$ fill 2-dimensional Liouville tori with frequencies satisfying (\ref{pq}). They represent trefoil cable knots in $S^3\setminus \mathcal K$ with parameters $p,q$, having linking number $l=6p$  with  $\mathcal K$.
 
 Any cable knot of trefoil can be realised in such a way. 
 \end{thm} 
 
 Thus in the integrable limit the class of modular knots is replaced by the class of the trefoil cable knots. In the Lorenz systems this change is shown on Fig. 5.
 
 It is instructive to see also what happens in the case when $$\Gamma=\Gamma_2/\{\pm I\}\subset PSL(2,\mathbb Z),$$  where $\Gamma_2 \subset SL(2,\mathbb Z)$ is the principal congruence subgroup, consisting of the matrices congruent to the identity modulo 2. The quotient surface $\mathcal M^2_{\Gamma_2}$ in this case is the sphere punctured at three points with the fundamental domain being the ideal triangle.
 The corresponding 3-fold can be represented topologically as the complement $$\mathcal M^3_{\Gamma_2}\cong S^3\setminus \mathcal L,$$ where $\mathcal L$ is the link of 3 Hopf fibres with pairwise linking numbers 1 shown on Fig. 7 (not to be confused with the famous Borromeo link with the pairwise linking numbers being zero).
 Note that in this case such a representation is not unique, since the topology of the complement to a link does not determine the link, in contrast to the knot case \cite{GL}.

The periodic geodesics for large $\mathcal C$ are cable knots of the Hopf fibres, which are simply torus knots in this case.
 \begin{figure}[h]
\includegraphics[width=60mm]{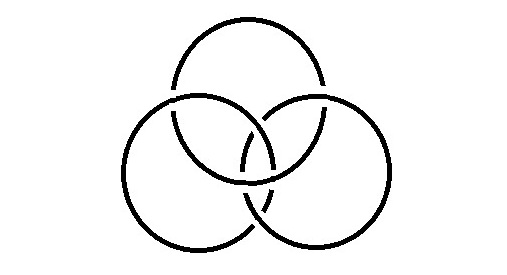}  
\caption{The Hopf 3-link $\mathcal L$}
\end{figure}

Recall that a knot $K \subset S^3$ is called {\it hyperbolic}, if its complement admits hyperbolic structure (complete constant negative curvature metric). Thurston \cite{Th2} proved that the only knots which are not hyperbolic are torus and satellite knots. 

Complements to satellite knots do not admit any geometric structure in Thurston's sense \cite{Th2, Cooper}, while for torus knots they admit $SL(2,\mathbb R)$-geometry. To see this we simply replace the modular surface with two orbifold points of order $2$ and $3$ by the hyperbolic surface $\mathcal M^2_{p,q}$ with one cusp and two orbifold points of order $p$ and $q.$ 

Not all modular knots are hyperbolic (although many of them are), but their links with the trefoil knot $\mathcal K$ are always hyperbolic, as it was shown by Foulon and Hasselblat \cite{FH}. The volumes of the corresponding hyperbolic complements of modular knots in $\mathcal M^3$ are studied in recent papers \cite{Bergeron, Brandts}. 

%
%
 
 \section{Topological point of view}
 
Let's discuss our problem from the topological point of view. The existence of a metric on a topological manifold $M^n$ with an integrable geodesic flow imposes some restriction on the topology of $M^n.$

The first such restrictions were found by Kozlov, who proved that two-dimensional closed oriented manifold $\mathcal M^2_g$ of genus $g > 1$ does not admit such metric with an analytic additional integral \cite{Kozlov}.

The multi-dimensional case was studied by Taimanov \cite{T}, who proved that if 
$M^n$ admits a geodesic flow which is Liouville integrable in real analytic sense, then
$\dim H_1(M^n, \mathbb R))\le n$ and the fundamental group  $\pi_1(M^n)$ is almost commutative (i.e. contains a commutative subgroup of finite index).  

In the case of the unit tangent bundle $\mathcal M^3=S\mathcal M^2_g$ of a genus $g$ surface $\mathcal M^2_g$
both Taimanov's conditions fail. Indeed, the fundamental group $\pi_1(\mathcal M^3)$  is generated by $a_1,\dots, a_g, b_1,\dots,b_g, \gamma$ with the relations
$$
a_i\gamma=\gamma a_i, b_i\gamma=\gamma b_i \,\, (i=1,\dots,g), \,\,\, \gamma^{2-2g}=(a_1b_1a_1^{-1}b_1^{-1})\dots (a_gb_ga_g^{-1}b_g^{-1})
$$
(see e.g. Chapter 1, Section 4 in \cite{DFN}). This implies that $\pi_1(\mathcal M^3)$ is not almost commutative with $\dim H_1(\mathcal M^3,\mathbb R)=2g>3=\dim \mathcal M^3,$ so by Taimanov's theorem there are no analytically integrable geodesic flows on $\mathcal M^3$, which agrees with our result.

This also agrees with Dinaburg's theorem \cite{Dinaburg}, claiming that if the fundamental group $\pi_1(M^n)$ of
a manifold $M^n$ has an exponential growth, then the topological entropy of
the geodesic flow of any Riemannian metric on $M^n$ is positive (see also Manning \cite{Manning}).

We should also mention an important result of Butler \cite{Butler2005}  which states that if a 3-fold $M^3$ admits an integrable geodesic flow 
with smooth (not necessarily analytic) integrals and a ``tame'' singular set,  then the fundamental group $\pi_1(M^3)$ contains a finite--index polycyclic subgroup of step length at most 4. This means that a closed $SL(2,\mathbb R)$-manifold does not admit integrable geodesic flows of this type even in smooth category.
For a review of other results on topological obstructions to integrability we refer to \cite{B}.

Note that the first examples of Liouville integrable geodesic flows with smooth integrals with positive topological entropy found by Bolsinov and Taimanov \cite{BT} are $Sol$-manifolds, which are topologically the torus bundles over circles, twisted by hyperbolic elements $A \in SL(2, \mathbb Z).$ It is interesting that they were also historically the first examples of 3-folds considered by Poincar\'e in 1892, see \cite{Stillwell}.

The most interesting question is to describe which type of knots may appear as periodic geodesics on the modular 3-fold for all values of the integral $\mathcal C.$ This question is most likely to be out of reach, but even to find possible topological restrictions would be very interesting. 

A similar question can be asked for the Lorenz system with arbitrary values of $r.$ Note that for large $r$ (and at least for $b$ not close to $\frac{3}{2}\sigma-\frac{1}{2}$) the system has at most 3 periodic orbits, each representing a trivial knot \cite{MN}. 

It is interesting that the trefoil cable knots (and more generally, iterated torus knots) appearing in the integrable region have also very interesting relation with the theory of double affine Hecke algebras, see the recent papers by Berest and Samuelson \cite{BS, Samuelson} and Cherednik and Danilenko \cite{CD}.

Note also that iterated torus knots are precisely the knots with zero topological entropy, understood as the minimal topological entropy of diffeomorphisms of a disc, whose mapping torus respects the knot (see Llibre and MacKay \cite{LM}).

\section{Acknowledgements}

We are very grateful to Dmitri Alekseevsky, Yuri Berest, Ivan Cherednik, Ivan Dynnikov, Anatoly Neishtadt, Graeme Segal, Caroline Series and John Parker for very helpful and stimulating discussions. 

We are also grateful to Jos Leys for kind permission to use his spectacular images of modular knots.

A.\,Bolsinov was supported by the Russian Science Foundation grant no. 17-11-01303.


\begin{thebibliography}{99}


\bibitem{Adams}
C.C. Adams {\it The Knot Book: An Elementary Introduction to the Mathematical Theory of Knots.} AMS, Providence, Rhode Island, 2004.

\bibitem{Anosov}
D.V. Anosov {\it Geodesic flows on closed Riemannian manifolds of negative curvature.} Proc. Steklov Math. Inst. {\bf 90} (1967), 1-235.

\bibitem{Arnold1961}
V.I. Arnold {\it Some remarks on flow of line elements and frames}. Sov. Math. Dokl. {\bf 138:2} (1961), 255--257.

\bibitem{Arnold}
V.I. Arnold {\it Mathematical Methods of Classical Mechanics}. Springer, 1989.

\bibitem{Artin}
E. Artin {\it  Ein mechanisches  System  mit  quasiergodischen  Bahnen.}   Abh.  Math.  Sem.  Univ. Hamburg  {\bf 3} (1924),  170--175.

\bibitem{BS}
Yu. Berest, P. Samuelson {\it Double affine Hecke algebras and generalized Jones polynomials.} Compositio Math. 152 (2016) 1333-1384.

\bibitem{Bergeron}
M. Bergeron, T. Pinsky, L. Silberman {\it An upper bound for the volumes of complements of periodic geodesics.}
Intern. Math. Res. Notices, Vol. 2017, 1-23.

\bibitem{BW}
J.S. Birman, R.F. Williams {\it Knotted periodic orbits in dynamical systems. I. Lorenz's equations.} Topology {\bf 22(1)} (1983), 47-82.

\bibitem{B}
A.V. Bolsinov {\it Integrable geodesic flows on Riemannian manifolds.} J. Math. Sci. {\bf 123} (2004), 4185-98.

\bibitem{BT} 
A.V. Bolsinov, I.A. Taimanov {\it Integrable geodesic flows with positive topological entropy.}
Invent. Math. {\bf 140} (2000), 639-650.

\bibitem{Brandts}
A. Brandts, T. Pinsky, L. Silberman {\it Volumes of hyperbolic three-manifolds associated to modular links.} arXiv 1705.04760, 2017. Symmetry, 2019, 1-9.

\bibitem{BP}
K. Burns, G.P. Paternain {\it Anosov magnetic flows, critical values and topological entropy.} Nonlinearity {\bf 15} (2002), 281-314.

\bibitem{Butler}
L. Butler {\it A new class of homogeneous manifolds with Liouville-integrable geodesic flows.} C. R. Math. Acad. Sci. Soc. R. Can. {\bf 21}(1999), no. 4, 127-131. 

\bibitem{Butler2005}
L.T.  Butler {\it Invariant fibrations of geodesic flows.} Topology {\bf 44:4} (2005), 769-789.

\bibitem{CD}
I. Cherednik and I. Danilenko {\it DAHA and iterated torus knots.} Alg. Geom. Topology {\bf 16} (2016), 843-898.



\bibitem{Bo}
F. Dal'Bo {\it Geodesic and Horocyclic Trajectories.} Universitext. Springer-Verlag London, 2011.

\bibitem{Dinaburg}
E.I. Dinaburg {\it On the relations among various entropy characteristics of dynamical
systems.} Math. USSR Izv. {\bf 5} (1971), 337-378.

\bibitem{Caratheodory}
C. Caratheodory {\it Conformal Representation.} Cambridge Tracts in Mathematics and Mathematical Physics, Camb. Univ. Press, 1932.

\bibitem{DC}
M.P. do Carmo {\it Differential Geometry of Curves and Surfaces.} Prentice-Hall, 1976.

\bibitem{Cooper}
D. Cooper,  C. Hodgson, S. Kerckhoff  {\it Three-Dimensional Orbifolds and Cone-Manifolds.}  MSJ Memoirs {\bf 5} (2000), Math. Soc. Japan, Tokyo.

\bibitem{DFN}
B.A. Dubrovin, A.T. Fomenko, S.P. Novikov {\it Modern Geometry - Methods and Applications: Part 3: Introduction to Homology Theory.} Springer Verlag, 1990.

\bibitem{FH}
P. Foulon, B. Hasselblat {\it Contact Anosov flows on hyperbolic 3-manifolds.}  Geometry and Topology, {\bf 17} (2013), 1225-52.

\bibitem{Ghys}
\'E. Ghys {\it Knots and Dynamics.} Intern. Congress of Math. Vol. 1. Eur. Math. Soc., Zurich 2007, 247-277.

\bibitem{GhL}
\'E. Ghys and J. Leys {\it Lorenz and Modular Flows: A visual Introduction.} www.ams.org/featurecolumn/archive/lorenz.html

\bibitem{GL}
C.McA. Gordon, J. Luecke {\it Knots are determined by their complements.} Bull. Amer. Math. Soc. (N.S.) {\bf 20} (1989), no. 1, 83-87.

\bibitem{GK}
B. Gurevich, S. Katok {\it Arithmetic coding and entropy for the positive geodesic flow on the modular surface.} Mosc. Math. J. {\bf 1:4} (2001), 569-582.

\bibitem{HI}
S. Halverscheid, A. Iannuzzi {\it On naturally reductive left-invariant metrics of $SL(2,\mathbb R)$.}
Ann. Scuola Norm. Sup. Pisa, CI. Sci. (5), Volume 5 (2006) no. 2, 171-187. 

\bibitem{Hedlund-1}
G.A. Hedlund {\it Fuchsian groups and transitive horocycles.} Duke Math. J. {\bf 2} (1936), no. 3, 530--542.

\bibitem{Helgason}
S. Helgason {\it Differential Geometry and Symmetric Spaces.} AMS Chelsea Publ., 1962.

\bibitem{Katok}
A. Katok {\it Fifty years of entropy in dynamics: 1958-2007.} J. Modern Dynamics, {1:4} (2007), 545-596.


\bibitem{SKatok}
S. Katok {\it Fuchsian groups.} The University of Chicago Press, 1992.

\bibitem{Klein}
F. Klein {\it Ausgew\"ahlte Kapitel der Zahlentheorie I. Vorlesung, gehalten im Wintersemester
1895/96. Ausgearbeitet von A. Sommerfeld.} G\"ottingen, 1896.

\bibitem{Kozlov}
V.V. Kozlov {\it Topological obstructions to the integrability of natural mechanical systems.}
Soviet Math. Dokl. {\bf 20} (1979), 1413-1415.

\bibitem{JL}
J. Llibre {\it Brief survey on the topological entropy.} Discrete Contin. Dyn. Syst. Ser. B, {20(10)} (2015), 3363-74.

\bibitem{LM}
J. Llibre, R.S. MacKay {\it A classification of braid types for diffeomorphisms of surfaces of genus zero with topological entropy zero.} 
J. London Math. Soc. {\bf 42} (1990), 562-576.

\bibitem{Lorenz}
E. Lorenz {\it Deterministic nonperiodic flow.} J. Atmos. Sci. {\bf 20} (1963), 130-141.

\bibitem{Manning}
A.  Manning {\it Topological     entropy    for    geodesic    flows.}    Ann.  Math. {\bf 110} (1979),  567-573. 

%
\bibitem{Mielke}
A. Mielke {\it Finite elastoplasticity, Lie groups and geodesics on $SL(d).$} In the book {\it Geometry, Mechanics and Dynamics.} Springer, New York, 2002, 61-90.

\bibitem{Milnor}
J. Milnor {\it Introduction to Algebraic $K$-theory.} Princeton University Press, 1972.

\bibitem{Miranda}
J.A.G. Miranda {\it Positive topological entropy for magnetic flows on surfaces.} Nonlinearity {\bf 20:8} (2007), 2007--2031.

\bibitem{MN}
A.V. Moiseev, A.I. Neishtadt {\it Phase portrait of the Lorenz system at large Rayleigh numbers.} Izv. Russ. Acad. Sci., ser. Mech. Solid Body, {\bf 4} (1995), 20--27.

\bibitem{Mont1995}
R. Montgomery {\it Survey of singular curves in sub-Riemannian geometry.} J. Dynam. Control Systems, {\bf 1:1} (1995), 49-90.

\bibitem{MT}
J.W. Morgan, G. Tian {\it The Geometrization Conjecture.} Clay Mathematics Monographs, {\bf 5}, AMS, 2014.

\bibitem{Mostovoy}
J. Mostovoy  {\it Lattices in $\mathbb C$ and finite subsets of a circle.}
The Amer. Math. Monthly {\bf 111, No. 4} (2004), 357-360. 

\bibitem{Nagy}
 P.T. Nagy {\it On the tangent sphere bundle of a Riemannian 2-manifold.} Tohoku Math. J. (2)
 {\bf 29:2} (1977), 203-208.
 
 \bibitem{Parker}
 J. Parker {\it Private communication.} August 2018.
 
 \bibitem{Pat}
G. Paternain {\it On the regularity of the Anosov splitting for twisted geodesic flows.} Mathematical Research Letters
{\bf 4:6} (1997), 871-888.

\bibitem{Pesin}
Ya.B. Pesin {\it  Characteristic Lyapunov exponents and smooth ergodic theory.} Russian Math. Surv. {\bf 32:4} (1977), 55-112.

\bibitem{Robbins}
K.A. Robbins {\it Periodic solutions and bifurcation structure at high $r$ in the Lorenz model.}  SIAM J. Appl. Math. {\bf 36(3)} (1979), 457-472.

\bibitem{Samuelson}
P. Samuelson {\it Iterated torus knots and double affine Hecke algebras.} Intern. Math. Res. Notices, Volume 2019, Issue 9, 2848-2893.



\bibitem{Sarnak}
P. Sarnak {\it Linking numbers of modular knots.} Comm. Math. Anal. {\bf 8}(2010), 136-144.

\bibitem{Sasaki}
S. Sasaki {\it On the differential geometry of tangent bundles of Riemannian manifolds.} Tohoku Math. J. {\bf (2) 10} (1958), no. 3, 338--354.

 \bibitem{Scott} P. Scott
{\it The geometries of 3-manifolds.} Bull. London Math. Soc., {\bf
15} (1983), 401-487.

\bibitem{Stillwell}
J. Stillwell {\it Poincar\'e and the early history of 3-manifolds.} Bull. Amer. Math. Soc. {\bf 49(4)} (2012), 555-576.


\bibitem{Strogatz}
S.H. Strogatz {\it Nonlinear Dynamics and Chaos.}  Addison-Wesley, 1994.

\bibitem{T} I.A. Taimanov {\it Topological obstructions to
integrability of geodesic flows on non-simply-connected manifolds.}
Math. USSR Izv. {\bf 30} (1988), 403-409.

\bibitem{T2004}
I.A. Taimanov {\it An example of a jump from chaos to integrability for magnetic geodesic flows.}
Math. Notes {\bf 76} (2004), 587-589. 

\bibitem{Th} W.P. Thurston
{\it Hyperbolic geometry and $3$-manifolds. }  
In: LMS Lecture Notes Series {\bf 48}, Cambridge Univ. Press, 1982. 

\bibitem{Th2} W.P. Thurston
{\it Three-dimensional manifolds, Kleinian groups and hyperbolic geometry.} 
Bull. Amer. Math. Soc. {\bf 6:3} (1982), 357--381.

 

\bibitem{Yudovich}
V.I. Yudovich {\it Asymptotics of limit cycles of the Lorenz system for large Rayleigh numbers.} (In Russian) Preprint, Rostov State Univ.; VINITI, N 2611-78, 1978.

%


%


\end{thebibliography}
\end{document}